\documentclass[12pt,a4paper]{article}
\usepackage{amsmath}
\usepackage{amssymb}
\usepackage{amsthm}

\newtheorem{defi}{Definition}[section]
\newtheorem{thm}[defi]{Theorem}
\newtheorem{prop}[defi]{Proposition}
\newtheorem{lemma}[defi]{Lemma}

\theoremstyle{remark}
\newtheorem{rem}[defi]{Remark}

\DeclareFontEncoding{OT2}{}{}
\DeclareFontSubstitution{OT2}{cmr}{m}{n}
\DeclareFontFamily{OT2}{cmr}{\hyphenchar\font45 }
\DeclareFontShape{OT2}{cmr}{m}{n}{%
	<5><6><7><8><9>gen*wncyr%
	<10><10.95><12><14.4><17.28><20.74><24.88>wncyr10}{}
\DeclareFontShape{OT2}{cmr}{b}{n}{%
	<5><6><7><8><9>gen*wncyb%
	<10><10.95><12><14.4><17.28><20.74><24.88>wncyb10}{}
\DeclareMathAlphabet{\mathcyr}{OT2}{cmr}{m}{n}
\DeclareMathAlphabet{\mathcyb}{OT2}{cmr}{b}{n}
\SetMathAlphabet{\mathcyr}{bold}{OT2}{cmr}{b}{n}

\title{Some remarks on
the order structures of 
multi-polylogarithms}
\date{\empty}
\author{Ken Kamano}

\begin{document}

\maketitle

\begin{abstract}

We consider multi-polylogarithm functions
which are slightly different from the ordinary ones.
These functions have two integral representations
and an order structure 
similar to those of multiple zeta star values.
We also give a
necessary and sufficient condition
for 
the set of values of our multi-polylogarithm functions
to be a dense set.
\end{abstract}

\section{Introduction}

We call a tuple $(k_1,\ldots ,k_r)$ of positive integers 
an \textit{index}.
Define the set $\mathcal{I}$ of all indices as 
\[ \mathcal{I}:=
\{ (k_1,\ldots, k_r) \mid r\ge 1, k_1,\ldots, k_r \ge 1\}\]
and its subset $\mathcal{S}$ as
\begin{align*}
\mathcal{S} &:=
\{ (k_1,\ldots, k_r) \mid r\ge 1, k_1\ge 2, k_2,\ldots, k_r \ge 1\}.
\end{align*}
For any index $\boldsymbol{k}
=(k_1,\ldots, k_r)\in \mathcal{S}$, 
\textit{multiple zeta star values} (MZSV)
are defined by
\[ \zeta^{\star}(k_1,\ldots ,k_r)
:=
 \sum_{m_1\ge \cdots \ge m_r>0}
\dfrac{1}{m_1^{k_1}\cdots m_r^{k_r}},
\]
which are analogues of \textit{multiple zeta
values} (MZV) defined by
\[ \zeta(k_1,\ldots ,k_r)
:=
 \sum_{m_1> \cdots > m_r>0}
\dfrac{1}{m_1^{k_1}\cdots m_r^{k_r}}.
\]
It is well-known that 
MZSV can be written as a
$\mathbb{Q}$-linear combination of MZV and 
vice versa, 
and they are represented by iterated integrals (see e.g., \cite{Zha}).
In \cite{Li} Li gave another integral representation
of MZSV. 

\begin{thm}[{\cite[Theorem 1.1 (ii)]{Li}}]
\label{thm:Li_integral}
Let $j_1,\ldots ,j_{2k}\in \mathbb{Z}_{>0}$ and 
$r=2k$.
Set $i_l:= j_1+\cdots +j_l$ $(1\le l \le 2k)$.
Then 
\begin{equation}\label{eq:Li's formula}
\begin{split}
&\displaystyle \int_{[0,1]^{i_r}}
\dfrac{1}{1-(x_1\cdots x_{i_1}) + (x_1\cdots x_{i_2})
- \cdots + (x_1\cdots x_{i_r}) } dx_1 \cdots dx_{i_r}\\
&=
\zeta^{\star}
(j_1+1, \overbrace{1,\ldots ,1}^{j_2-1},
\ldots ,j_{2k-1}+1, \overbrace{1,\ldots ,1}^{j_{2k}-1}).
\end{split}
\end{equation}
\end{thm}

\begin{rem}
    It is worth mentioning that this formula was essentially proved by 
Zlobin (see \cite[Lemma 2]{Zlo}).
\end{rem}

The order structures of MZV 
and MZSV have been studied by
Kumar \cite{Ku} and Li
\cite{Li}, respectively.
Define an order $\succ$ on $\mathcal{S}$
by
\[ (k_1, \ldots, k_r, k_{r+1})
\succ (k_1, \ldots, k_r)
\]
for any $(k_1, \ldots, k_r, k_{r+1}) \in \mathcal{S}$
and 
\[ (k_1, \ldots, k_r)
\succ (m_1, \ldots, m_s)
\]
if $k_i=m_i$ ($1\le i \le j-1$) and 
$k_j<m_j$ for some $j\ge 1$. 
By using Theorem \ref{thm:Li_integral}, Li proved the following theorem.
\begin{thm}[{\cite[Theorem 1.2]{Li}}]
For any 
$ (k_1, \ldots, k_r)$ and 
$(m_1, \ldots, m_s) \in \mathcal{S}$,
$ (k_1, \ldots, k_r)\succ (m_1, \ldots, m_s) $
if and only if 
\[ \zeta^{\star}(k_1,\ldots ,k_r)
> \zeta^{\star}(m_1,\ldots ,m_s). \]
\end{thm}

Let 
\[
\mathcal{T}:=
\{ (k_1,k_2,\ldots )\in \mathbb{Z}_{>0}^{\infty} \mid
k_1\ge 2, \ k_s\ge 2 \text{ for some } s\ge 2 \text{ if } k_1=2
\} .
\]
One can extend the order on $\mathcal{S}$
naturally to an unique order on $\mathcal{T}$.
Then the following theorem holds.
\begin{thm}[{\cite[Theorem 1.3]{Li}}]
A map
\begin{align*}
\eta: \mathcal{T} &\longrightarrow (1, +\infty) \\
\boldsymbol{k} = (k_1,k_2,\ldots ) &\longmapsto 
\lim_{r\to\infty} \zeta^{\star}(k_1,k_2, \ldots , k_r)
\end{align*}
is bijective
and $\boldsymbol{k}\succ \boldsymbol{m}$
if and only if $\eta(\boldsymbol{k})> \eta(\boldsymbol{m})$.
\end{thm}
As a corollary of this theorem, Li also proved that 
the set $\{ \zeta^{\star}(\boldsymbol{k}) 
\mid \boldsymbol{k} \in \mathcal{S} \}$
is dense in $[1, \infty)$.
We remark that in \cite{HMO}
Hirose, Murahara and Onozuka
independently investigated MZSV of infinite length
$\zeta(\boldsymbol{k})$ with $\boldsymbol{k}\in
\mathcal{T}$ and 
gave some values explicitly.
\vspace{10pt}

Throughout this paper, 
we assume that 
$\boldsymbol{k}=(k_1,\ldots, k_r)\in
\mathbb{Z}_{>0}^r$ and $\boldsymbol{z}=(z_1,z_2,\ldots ,z_r)$
satisfy
\[ 1\ge z_1\ge z_2 \ge \cdots \ge z_r > 0\]
with $(k_1,z_1)\neq (1,1)$.
As an analogue of MZSV,
the following
multi-polylogarithms of shuffle type can be considered:
\begin{align}\label{eq:Li_shuffle}
\textrm{Li}_{\boldsymbol{k}}^{\mathcyr{sh},\star} (z_1, 
\ldots, z_r)
&:=
\sum_{m_1 \ge \cdots \ge m_r\ge 1}
 \dfrac{z_1^{m_1} (z_2/z_1)^{m_2}\cdots
(z_{r}/z_{r-1})^{m_r}}
{m_1^{k_1}\cdots m_r^{k_r}}.
\end{align} 
For any index $\boldsymbol{k}
\in \mathbb{Z}_{>0}^r$, 
we introduce the following function,
which is slightly different from 
$\textrm{Li}_{\boldsymbol{k}}^{\mathcyr{sh},\star}$
defined by \eqref{eq:Li_shuffle}:
\begin{align*}
l^{\star}
\begin{pmatrix}
k_1, \ldots , k_r \\
z_1, \ldots , z_r
\end{pmatrix}
&:=
\dfrac{1}{z_r}
\textrm{Li}_{\boldsymbol{k}}^{\mathcyr{sh},\star} (z_1, 
\ldots, z_r)\\
&=\sum_{m_1\ge \cdots \ge m_r \ge 1} 
\dfrac{z_1^{m_1-1} (z_2/z_1)^{m_2-1}\cdots
(z_{r}/z_{r-1})^{m_r-1}}
{m_1^{k_1}\cdots m_r^{k_r}}.
\end{align*}
For $k_1\ge 2$, the value
$\textrm{Li}_{\boldsymbol{k}}^{\mathcyr{sh},\star}
(1, \ldots, 1)=l^{\star}
\begin{pmatrix}
k_1, \ldots , k_r \\
1, \ldots , 1
\end{pmatrix}
$ is nothing but the 
multiple zeta star value
$\zeta^{\star}(k_1,\ldots, k_r)$.
As we will see in the next section,
the function $\textrm{Li}_{\boldsymbol{k}}^{\mathcyr{sh},\star}$ does not have an order structure similar to MZSV, but the function $l^{\star}$ does.

The present paper is organized as follows:
in Section \ref{Multiple polylogarithms}, 
we give two integral representations of 
the function $l^{\star}$
and show that it has an order structure 
similar to MZSV.
In Section \ref{Multiple polylogarithms of infinite length}, 
we introduce multiple polylogarithms of infinite length,
and give a necessary and
sufficient condition for the set of our multi-polylogarithm functions to be a dense set.

\section{Multiple polylogarithms}\label{Multiple polylogarithms}

We first see the iterated integral representation of 
the function $l^{\star}$. Let
\[ \omega_0(t)=\dfrac{1}{t},\ \ 
\omega_i(t)=\dfrac{1}{t(1-z_it)}\ (1\le i\le r-1),\ \ 
\omega_r(t)=\dfrac{1}{1-z_rt}.\] 
Define
\begin{align*}
I(\varepsilon_1, \ldots ,\varepsilon_k)
:=\displaystyle \int \cdots \int 
_{1>t_1>\cdots >t_k>0}
\omega_{\varepsilon_1}(t_1) \cdots 
\omega_{\varepsilon_k}(t_k),
\end{align*}
where each $\varepsilon_i \in \{ 0, 1, \ldots , r\}$.
By direct calculation, one can obtain the following 
iterated integral representation of $l^{\star}$.
\begin{prop}
For $(k_1, \ldots, k_r)\in \mathbb{Z}_{>0}^r$, we have

\begin{align}\label{eq:int_rep}
l^{\star}
\begin{pmatrix}
k_1, \ldots , k_r \\
z_1, \ldots , z_r
\end{pmatrix}
=
I(\overbrace{0,\ldots,0}^{k_1-1},\, 1,\, \overbrace{0,\ldots,0}^{k_2-1},
\, 2,\, 
\ldots , \overbrace{0,\ldots,0}^{k_r-1}, \,r \,).
\end{align}
\end{prop}
Since $\omega_i(t)
= \frac{1}{t}+\frac{1}{1-t}$ 
if $z_i=1$ ($1\le i \le r-1$), 
the equation \eqref{eq:int_rep}
for $k_1\ge 2$ and $z_1=\cdots =z_r=1$
coincides with
the iterated integral representation of MZSV.

The function $l^{\star}$ has 
the following another
integral representation as well as MZSV. 
When $k_1\ge 2$ and 
all $z_i=1$ ($1\le i \le r$), 
this representation coincides with 
Li's formula \eqref{eq:Li's formula}.
For an index $\boldsymbol{k}=(k_1,\ldots, k_r)\in
\mathbb{Z}_{>0}^r$,
set $K_i:=k_1+\cdots +k_i$ ($1\le i \le r$)
and $k:=K_r=k_1+\cdots +k_r$.
For $1\le i \le r$, we define $P_i((k_1, \ldots, k_r),(z_1,\ldots ,z_r);(x_1,\ldots, x_{k}))$,
simply denoted by $P_i$, as 
\[ P_i:=x_1\cdots x_{K_i-1}(1-x_{K_i})z_i \ \ (1\le i \le r).\]

Then another integral representation of 
$l^{\star}$ is as follows:
\begin{thm}\label{thm:integralprop}
For an index $\boldsymbol{k}=(k_1,\ldots, k_r)\in
\mathbb{Z}_{>0}^r$, 
we have
\begin{equation}\label{eq:anotherint_exp_l}
l^{\star}
\begin{pmatrix}
k_1, \ldots , k_r \\
z_1, \ldots , z_r
\end{pmatrix}= \displaystyle\int_{[0,1]^k}
\dfrac{1}{1- \sum_{i=1}^r P_i}\, d\boldsymbol{x},
\end{equation}
where 
$d\boldsymbol{x}=dx_1 dx_2\cdots dx_{k}$
\end{thm}

To prove Theorem \ref{thm:integralprop}, we need the 
following lemma, which is a special case of 
\cite[Lemma 2.1]{Li}.
\begin{lemma}\label{lemma:integration_lemma}
For $\alpha\in\mathbb{C}$ and $m \in\mathbb{Z}_{>0}$, we have
\[ \int_0^1 (1-x(1-\alpha))^{m-1} \, dx
=  \frac{1}{m}\sum_{m_1=1}^m \alpha^{m_1-1}.\]  
\end{lemma}

\begin{proof}[Proof of Theorem \ref{thm:integralprop}]
When $r=1$,
the right-hand side of 
\eqref{eq:anotherint_exp_l} equals 
\begin{align*}
\int_{[0,1]^{k_1}}
\frac{1}{1-x_1\cdots x_{k_1-1}(1-x_{k_1})z_1}\, dx_1\cdots 
dx_{k_1}
&=
\int_{[0,1]^{k_1}}
\frac{1}{1-x_1\cdots x_{k_1-1}x_{k_1}z_1}\, dx_1\cdots 
dx_{k_1}\\
&=
\int_{[0,1]^{k_1}}
\sum_{m=1}^{\infty}
(x_1\cdots x_{k_1-1}x_{k_1}z_1)^{m-1}\, dx_1\cdots 
dx_{k_1}\\
&=\sum_{m=1}^{\infty} \frac{z_1^{m-1}}{m^{k_1}}
\end{align*}
and the statement holds true.

When $r\ge 2$ we have
\begin{align*}
\dfrac{1}{1- \sum_{i=1}^r P_i}
&=
\sum_{m_1=1}^{\infty} 
\left(\sum_{i=1}^rP_i\right)^{m_1-1}\\
&=
\sum_{m_1=1}^{\infty} 
(x_1\cdots x_{k_1-1}z_1)^{m_1-1} \\
&\hspace{40pt}\times \Bigg( 
(1-x_{k_1}) + \left(x_{k_1}\cdots x_{K_2-1}(1-x_{K_2}) \tfrac{z_2}{z_1}\right)\\
&\hspace{100pt} +\cdots +\left(x_{k_1}\cdots x_{K_r-1}(1-x_{K_r}) \tfrac{z_r}{z_1}\right)
\Bigg)^{m_1-1}.
\end{align*}
By integrating both sides with respect to
$x_1, \ldots ,x_{k_1-1}$, we have
\begin{align*}
\int_{[0,1]^{k_1-1}} \dfrac{1}{1- \sum_{i=1}^r P_i}\,dx_1 \cdots dx_{k_1-1}
=
\sum_{m_1=1}^{\infty} 
\frac{z_1^{m_1-1}}{m_1^{k_1-1}} \Big\{ 1-x_{k_1} (1- E) \Big\}^{m_1-1},
\end{align*}
where 
\begin{align*}
E= x_{k_1+1}\cdots x_{K_2-1}(1-x_{K_2}) \tfrac{z_2}{z_1}
+\cdots 
+x_{k_1+1}\cdots x_{K_r-1}(1-x_{K_r}) \tfrac{z_r}{z_1}.
\end{align*}
By using Lemma \ref{lemma:integration_lemma},
we have
\begin{align*}
\int_{[0,1]^{k_1}} \dfrac{1}{1- \sum_{i=1}^r P_i}\,dx_1 \cdots dx_{k_1}=
\sum_{m_1=1}^{\infty} 
\frac{z_1^{m_1-1}}{m_1^{k_1}}
\sum_{m_2=1}^{m_1} E^{m_2-1}.
\end{align*}
Since the term $E$ is the form of
\[  \sum_{i=1}^{r-1} P_i\left((k_2,\ldots ,k_r),
   \left( \tfrac{z_2}{z_1},\ldots ,\tfrac{z_r}{z_1}\right);
   (x_{K_1+1}, \ldots ,x_{K_r})
\right),\]  
we can repeat this procedure 
and the equation
\[
\displaystyle\int_{[0,1]^k}
\dfrac{1}{1- \sum_{i=1}^rP_i} d\boldsymbol{x}=
\sum_{m_1=1}^{\infty} 
\frac{z_1^{m_1-1}}{m_1^{k_1}}
\sum_{m_2=1}^{m_1}
\frac{(\frac{z_2}{z_1})^{m_2-1}}{m_2^{k_2}}
\cdots 
\sum_{m_r=1}^{m_{r-1}}
\frac{(\frac{z_r}{z_{r-1}})^{m_r-1}}{m_r^{k_r}}
\]
is eventually obtained.
The right-hand side equals $l^{\star}
\begin{pmatrix}
k_1, \ldots , k_r \\
z_1, \ldots , z_r
\end{pmatrix}$ and this completes the proof.
\end{proof}

By shifting terms, we have
$\sum_{i=1}^r P_i = \sum_{i=0}^rQ_i$ where
\begin{align*}
 \begin{cases}
    Q_0 =x_1\cdots x_{k_1-1}z_1,\\
    Q_i =-x_1\cdots x_{K_i}z_i +x_1\cdots x_{K_{i+1}-1}z_{i+1}
    \ (1\le i \le r-1),\\
    Q_r =-x_1\cdots x_{K_r}z_r,
 \end{cases}
\end{align*}
and obtain the expression
 \begin{equation*}
l^{\star}
 \begin{pmatrix}
 k_1, \ldots , k_r \\
 z_1, \ldots , z_r
 \end{pmatrix}
 = \displaystyle\int_{[0,1]^k}
 \dfrac{1}{1- \sum_{i=0}^r Q_i} d\boldsymbol{x}.
 \end{equation*}
This is used in the proof of the following lemma.

\begin{lemma}\label{lemma:1^m_inequality}
For $(k_1,\ldots, k_r) \in \mathbb{Z}_{>0}^r$ with $k_r\ge 2$
and $m\ge 1$, we have
\[ l^{\star}
\begin{pmatrix}
k_1, \ldots , k_r,& \  \overbrace{1\ ,\  \ldots \  ,\ 1}^m \\
z_1, \ldots , z_r,& z_{r+1}, \ldots ,z_{r+m}
\end{pmatrix}
< l^{\star}\begin{pmatrix}
k_1, \ldots , k_{r-1},& k_r-1 \\
z_1, \ldots , z_{r-1},& z_r
\end{pmatrix}.
\]

\end{lemma}
\begin{proof}

When
$\boldsymbol{k}
=(k_1, \ldots , k_r, \  \overbrace{1, \ldots , 1}^m)$,
we have 
\[K_i=K_{i+1}-1\ \ (r\le i \le r+m-1).\]
Hence, 
for $0\le x_i \le 1$ \  ($1\le i \le K_{r+m}$),
the values
\begin{align*}
 Q_r&= x_1\cdots x_{K_r}(-z_{r+1}+z_r),\\
\vdots \\
 Q_{r+m-1}&= x_1\cdots x_{K_{r+m-1}}(-z_{r+m-1}+z_{r+m}),\\
 Q_{r+m}&= -x_1\cdots x_{K_{r+m}}z_{r+m}
\end{align*}
are all less than or equal to zero
and we have
\[ \sum_{i=0}^{r+m} Q_i \le \sum_{i=0}^{r-1} Q_i.  \]
By setting $x_{K_r-1} \mapsto 1-x_{K_r-1}$,
the value
$Q_{r-1} = -x_1\cdots x_{K_{r-1}}z_{r-1}
+x_1\cdots x_{K_r-1}z_r
$
is transformed as 
\begin{align*}
\overline{Q}_{r-1}&= -x_1\cdots x_{K_{r-1}}z_{r-1}
+x_1\cdots x_{K_r-2}(1-x_{K_r-1})z_r \\
&= -x_1\cdots x_{K_{r-1}}z_{r-1}
+x_1\cdots x_{K_r-2}z_r\\
&\ \ \ \ - x_1\cdots x_{K_r-1}z_r.
\end{align*}
Therefore 
\begin{align*}
\displaystyle\int \frac{1}{1-\sum_{i=0}^{r+m} Q_i}\,
d\boldsymbol{x}
<
\int \frac{1}{1-\sum_{i=0}^{r-1} Q_i}\,
d\boldsymbol{x}
=
\int \frac{1}{1-\sum_{i=0}^{r-2} Q_i - \overline{Q}_{r-1}}\,
d\boldsymbol{x}
\end{align*}
(an empty sum stands for zero)
and the last term is an integral form of
$l^{\star}\begin{pmatrix}
k_1, \ldots , k_{r-1},& k_r-1 \\
z_1, \ldots , z_{r-1},& z_r
\end{pmatrix}$.
\end{proof}

We extend the order $\succ$ on $\mathcal{S}$
to the one on $\mathcal{I}$ by the same rules.
Similar to MZSV,
the order $\succ$ on $\mathcal{I}$
preserves the order of the values of $l^{\star}$.
\begin{thm}
For any 
$\boldsymbol{k}=(k_1, \ldots, k_r)\in \mathbb{Z}_{>0}^r$ and 
$\boldsymbol{m}=(m_1, \ldots, m_s)\in \mathbb{Z}_{>0}^s$,
$\boldsymbol{k}\succ \boldsymbol{m}$
if and only if 
\begin{equation}\label{eq:lstar_orderpres}
    l^{\star}\begin{pmatrix}
k_1, \ldots , k_r\\
z_1, \ldots , z_r
\end{pmatrix}
> l^{\star}\begin{pmatrix}
m_1, \ldots , m_s\\
z_1, \ldots , z_s
\end{pmatrix}.
\end{equation} 
\end{thm}

\begin{proof}
Assume that $\boldsymbol{k}\succ \boldsymbol{m}$.
When $r>s$ and $\boldsymbol{m}=(k_1,\ldots ,k_i)$
for some $i$ $(1\le i \le r-1)$,
obviously Eq.~\eqref{eq:lstar_orderpres} holds by definition.

When $\boldsymbol{m}=(k_1,\ldots ,k_{j-1},
m_j,\ldots , m_s)$
with $k_j<m_j$ for some $j$ $(1\le j \le \min\{r,s\})$,
by Lemma \ref{lemma:1^m_inequality},
we have
 \begin{align*}
 l^{\star}\begin{pmatrix}
  k_1, \ldots , k_r\\
  z_1, \ldots , z_r
 \end{pmatrix}
&\ge
 l^{\star}\begin{pmatrix}
  k_1, \ldots, k_{j-1}, k_j\\
  z_1, \ldots, z_{j-1},  z_j
 \end{pmatrix}\\
&>
 l^{\star}\begin{pmatrix}
  k_1, \ldots, k_{j-1}, k_j+1, \ \ 1,\ldots , 1\\
  z_1, \ldots, z_{j-1}, \ \  \ z_j, \ \ \ z_{j+1}, \ldots ,z_s
 \end{pmatrix}\\
&\ge
 l^{\star}\begin{pmatrix}
  k_1, \ldots, k_{j-1}, m_j, \ldots , m_s \\
  z_1, \ldots, z_{j-1}, \ z_j, \ldots , z_s
 \end{pmatrix}\\
&=
l^{\star}\begin{pmatrix}
m_1, \ldots , m_s\\
z_1, \ldots , z_s
\end{pmatrix}.
\end{align*}
Consequently
Eq. \eqref{eq:lstar_orderpres} also holds
and the theorem is proved.
\end{proof}

\begin{rem}
For the function
${\rm Li}_{\boldsymbol{k}}^{\mathcyr{sh},\star}$,
the statement of the theorem 
does not hold in general.
In fact, 
\begin{align*}
{\rm Li}_{2,1}^{\mathcyr{sh},\star} \left(\tfrac{2}{3}, \tfrac{1}{3}\right)
&=\sum_{m_1=1}^{\infty} \frac{(\frac{2}{3})^{m_1}}{m_1^2} \sum_{m_2=1}^{m_1}\frac{ (\frac{1}{2})^{m_2}}{m_2}\\
&< \sum_{m_1=1}^{\infty} \frac{(\frac{2}{3})^{m_1}}{m_1^2}\\
&={\rm Li}_{2}^{\mathcyr{sh},\star} \left(\tfrac{2}{3}\right)
\end{align*}
because of 
$\sum_{m_2=1}^{m_1}\frac{ (1/2)^{m_2}}{m_2}
<\sum_{m_2=1}^{\infty}\left(\frac{1}{2}\right)^{m_2}
=1$.
Therefore we have 
$(2,1) \succ (2)$ but 
${\rm Li}_{2,1}^{\mathcyr{sh},\star} \left(\tfrac{2}{3}, \tfrac{1}{3}\right)
< {\rm Li}_{2}^{\mathcyr{sh},\star} \left(\tfrac{2}{3}\right)$.
\end{rem}

\section{Multiple polylogarithms of infinite length}
\label{Multiple polylogarithms of infinite length}

For an integer $n\ge 1$, 
the notation $\{a\}^n$ means that 
the value $a$ is repeated $n$ times.
For example,
$l^{\star}
\begin{pmatrix}
2, \{1\}^{3}  \\
1, \{\tfrac{1}{2}\}^{3}  
\end{pmatrix}$
means $l^{\star}
\begin{pmatrix}
2, 1,1,1\\
1, \tfrac{1}{2}, \tfrac{1}{2},\tfrac{1}{2}  
\end{pmatrix}$.

\begin{prop}\label{prop:limit_prop}
\begin{enumerate}
\item 
For $0<z<1$, we have
$\displaystyle \lim_{n\to \infty}    
l^{\star}
\begin{pmatrix}
\{1\}^{n}  \\
\{z\}^{n}  
\end{pmatrix}
=\dfrac{1}{1-z}$.

\item 
For $k_r \ge 2$, we have
$\displaystyle \lim_{n\to \infty}    l^{\star}\begin{pmatrix}
k_1,\ldots , k_r, &\{1\}^n \\
z_1, \ldots, z_r, &\{z_r\}^n
\end{pmatrix}
=l^{\star}\begin{pmatrix}
k_1,\ldots , k_{r-1}, k_r-1\\
z_1, \ldots, z_{r-1}, \ z_r
\end{pmatrix}
$.

\item If $z_r> z_{r+1}$ then we have
\begin{align*}
&\displaystyle \lim_{n\to \infty}    l^{\star}\begin{pmatrix}
k_1,\ldots , k_r, &\{1\}^n \\
z_1, \ldots, z_r, &\{z_{r+1}\}^n
\end{pmatrix}\\
&= 
\dfrac{z_r}{z_r-z_{r+1}}
 l^{\star}\begin{pmatrix}
k_1,\ldots , k_r\\
z_1, \ldots, z_r
\end{pmatrix}
-\dfrac{z_{r+1}}{z_r-z_{r+1}}
l^{\star}\begin{pmatrix}
k_1,\ldots, k_{r-1}, k_r\\
z_1, \ldots, z_{r-1}, z_{r+1}
\end{pmatrix}.
\end{align*} 
\end{enumerate}

\end{prop}

\begin{proof}
We first show the equation
\begin{equation}\label{eq:limit_lemma}
\begin{split}
& \lim_{n\to \infty} l^{\star}
\begin{pmatrix}
k_1,\ldots , k_r, \ \ \{1\}^n\ \ \\
z_1, \ldots, z_r, \{z_{r+1}\}^n
\end{pmatrix} \\
&=
\sum_{m_1\ge \cdots \ge m_{r}\ge 1}
\frac{z_1^{m_1-1} (\frac{z_2}{z_1})^{m_2-1} \cdots (\frac{z_{r}}{z_{r-1}})^{m_{r}-1}}
     {m_1^{k_1} \cdots m_r^{k_r} }
\sum_{m_{r+1}=1}^{m_r}
\left(\frac{z_{r+1}}{z_{r}}\right)^{m_{r+1}-1}.
\end{split}
\end{equation}
The proof is similar to that of \cite[Lemma 3.4]{Li}.
For $n\ge 1$, we have
    \begin{align*}
&l^{\star}
\begin{pmatrix}
k_1,\ldots , k_r, \ \ \{1\}^n\ \ \\
z_1, \ldots, z_r, \{z_{r+1}\}^n
\end{pmatrix} \\
&=
\sum_{m_1\ge \cdots \ge m_{r+n}\ge 1}
\frac{z_1^{m_1-1} (\frac{z_2}{z_1})^{m_2-1} \cdots 
       (\frac{z_{r}}{z_{r-1}})^{m_{r}-1}}
     {m_1^{k_1} \cdots m_r^{k_r}  }
\frac{(\frac{z_{r+1}}{z_r})^{m_{r+1}-1}}{m_{r+1}}
\frac{1}{m_{r+2}} \cdots \frac{1}{m_{r+n}}\\
&=
1+\sum_{m_1\ge 2}\frac{z_1^{m_1}}{m_1^{k_1}}
+\cdots + 
\sum_{m_1\ge \cdots \ge m_{r+1}\ge 2}
\frac{z_1^{m_1-1} (\frac{z_2}{z_1})^{m_2-1} \cdots 
       (\frac{z_{r}}{z_{r+1}})^{m_{r+1}-1}}
     {m_1^{k_1} \cdots m_r^{k_r} m_{r+1} }\\
&\times  
\left(  1+ \sum_{m_{r+1}\ge m_{r+2}\ge 2}
   \frac{1}{m_{r+2}} + \cdots + 
   \sum_{m_{r+1}\ge \cdots \ge m_{r+n}\ge 2}
   \frac{1}{m_{r+2}\cdots m_{r+n} }
\right).
\end{align*}
By taking $n\to \infty$, we have 
\begin{align*}
& \lim_{n\to \infty} l^{\star}
\begin{pmatrix}
k_1,\ldots , k_r, \ \ \{1\}^n\ \ \\
z_1, \ldots, z_r, \{z_{r+1}\}^n
\end{pmatrix} \\
&=
1+\sum_{m_1\ge 2}\frac{z_1^{m_1}}{m_1^{k_1}}
+\cdots \\
&\ \ \ \ 
+\sum_{m_1\ge \cdots \ge m_{r+1}\ge 2}
\frac{z_1^{m_1-1} (\frac{z_2}{z_1})^{m_2-1} \cdots 
       (\frac{z_{r+1}}{z_{r}})^{m_{r+1}-1}}
     {m_1^{k_1} \cdots m_r^{k_r} m_{r+1} }  
\prod_{m_{r+1}\ge l \ge 2} 
\left(  1+\frac{1}{l}+\frac{1}{l^2}+\cdots \right).
\end{align*}
By using the identity
\[ \prod_{m_{r+1}\ge l \ge 2} 
\left(  1+\frac{1}{l}+\frac{1}{l^2}+\cdots \right)\\
=\prod_{m_{r+1}\ge l \ge 2} 
\frac{l}{l-1}
=m_{r+1},\]
we obtain that 
\begin{align*}
 \lim_{n\to \infty} l^{\star}
\begin{pmatrix}
k_1,\ldots , k_r, \ \ \{1\}^n\ \ \\
z_1, \ldots, z_r, \{z_{r+1}\}^n
\end{pmatrix}
=
\sum_{m_1\ge \cdots \ge m_{r+1}\ge 1}
\frac{z_1^{m_1-1} (\frac{z_2}{z_1})^{m_2-1} \cdots 
       (\frac{z_{r+1}}{z_{r}})^{m_{r+1}-1}}
     {m_1^{k_1} \cdots m_r^{k_r} }
    \end{align*}
and this proves \eqref{eq:limit_lemma}.

We apply
$r=1$, $k_1=1$ and $z=z_1=z_2$ with $0<z<1$
in \eqref{eq:limit_lemma}.
Then 
\[ \displaystyle \lim_{n\to \infty}    
l^{\star}
\begin{pmatrix}
\{1\}^{n}  \\
\{z\}^{n}  
\end{pmatrix}
=
\sum_{m_1=1}^{\infty} 
\frac{z^{m_1-1}}{m_1} \sum_{m_2=1}^{m_1}
1
= \dfrac{1}{1-z}\]
and the statement (i) is proved.
The statement (ii) can be proved similarly.

When $z_r> z_{r+1}$, by \eqref{eq:limit_lemma}
we have 
\begin{align*}
& \lim_{n\to \infty} l^{\star}
\begin{pmatrix}
k_1,\ldots , k_r, \ \ \{1\}^n\ \ \\
z_1, \ldots, z_r, \{z_{r+1}\}^n
\end{pmatrix} \\
&=
\sum_{m_1\ge \cdots \ge m_{r}\ge 1}
\frac{z_1^{m_1-1} (\frac{z_2}{z_1})^{m_2-1} \cdots 
       (\frac{z_{r}}{z_{r-1}})^{m_{r}-1}}
     {m_1^{k_1} \cdots m_r^{k_r} } \cdot 
\frac{1- ( \frac{z_{r+1}}{z_r})^{m_r}}
     {1- \frac{z_{r+1}}{z_r}} \\
&=\dfrac{z_r}{z_r-z_{r+1}}
 l^{\star}\begin{pmatrix}
k_1,\ldots , k_r\\
z_1, \ldots, z_r
\end{pmatrix}
-\dfrac{z_{r+1}}{z_r-z_{r+1}}
l^{\star}\begin{pmatrix}
k_1,\ldots, k_{r-1}, k_r\\
z_1, \ldots, z_{r-1}, z_{r+1}
\end{pmatrix}
\end{align*}
and this proves (iii).
\end{proof}

In the following
we assume that
$\boldsymbol{z}=(z_1, z_2, \ldots ) \in \mathbb{R}^{\infty}$ satisfies the condition
\[ 1> z_1 \ge z_2 \ge \cdots .\]
One can consider the case $z_1=1$ when $k_1\ge 2$,
but we do not treat this case for simplicity.
For 
an index 
$\boldsymbol{k}=(k_1, k_2, \ldots ) \in
\mathbb{Z}_{>0}^{\infty}$,
we define multiple polylogarithms of infinite length as
\begin{align}\label{eq:l^star_infinite}
 l^{\star}\begin{pmatrix}
\boldsymbol{k} \\
\boldsymbol{z} 
\end{pmatrix}
:=
\lim_{r\to \infty} 
 l^{\star}\begin{pmatrix}
k_1,\ldots, k_r \\
z_1,\ldots, z_r 
\end{pmatrix}.
\end{align}
Because
\[  l^{\star}\begin{pmatrix}
k_1,\ldots, k_r \\
z_1,\ldots, z_r 
\end{pmatrix}
< l^{\star}\begin{pmatrix}
\{1\}^r \\
\{z_1\}^r 
\end{pmatrix}
\to \dfrac{1}{1-z_1}\]
as $r$ tends to infinity,
the right-hand side of 
\eqref{eq:l^star_infinite} is convergent for 
any $\boldsymbol{k}=(k_1, k_2, \ldots )\in \mathbb{Z}_{>0}^{\infty}$.
For 
$\boldsymbol{k}=(k_1,k_2,\ldots)$ and
$\boldsymbol{m}=(m_1,m_2,\ldots)
\in \mathbb{Z}_{>0}^{\infty}$,
we denote 
$\boldsymbol{k}\succ \boldsymbol{m}$
if 
$(k_1,\ldots ,k_r) \succ
(m_1,\ldots ,m_r)$ for some $r\ge 1$.
This order defines 
a strict totally order on $\mathbb{Z}_{>0}^{\infty}$.

Define a map $\eta_{\boldsymbol{z}}$ as
\begin{align*}
    \eta_{\boldsymbol{z}} : \mathbb{Z}_{>0}^{\infty} &\longrightarrow 
    (1, \tfrac{1}{1-z_1}] \\
    \boldsymbol{k} &\longmapsto 
l^{\star}\begin{pmatrix}
\boldsymbol{k}\\
\boldsymbol{z}
\end{pmatrix}.
    \end{align*}

\begin{thm}
For 
$\boldsymbol{k}$,
$\boldsymbol{m}\in \mathbb{Z}_{>0}^{\infty}$,
it holds that $\eta_{\boldsymbol{z}}(\boldsymbol{k}) > \eta_{\boldsymbol{z}}(\boldsymbol{m})$
if and only if $\boldsymbol{k} \succ \boldsymbol{m}$.
In particular, 
the map $\eta_{\boldsymbol{z}}$ is injective.
\end{thm}

\begin{proof}
Because $(\mathbb{Z}_{>0}^{\infty}, \succ)$
is a strict totally ordered set,
we only have to prove that $\boldsymbol{k} \succ \boldsymbol{m}$
implies
$\eta_{\boldsymbol{z}}(\boldsymbol{k}) >
\eta_{\boldsymbol{z}}(\boldsymbol{m})$.

When $\boldsymbol{k} \succ \boldsymbol{m}$, 
these indices can be expressed as
$\boldsymbol{k}=(k_1,\ldots, k_{s-1}, k_s, \ldots )$
and $\boldsymbol{m}=(k_1,\ldots, k_{s-1}, m_s \ldots )$
with $k_s<m_s$ for some $s\ge 1$.
Then, by Proposition 
\ref{prop:limit_prop} (ii), we have
\begin{align*}
    l^{\star}\begin{pmatrix}
k_1,\ldots, k_s, \ldots \\
z_1,\ldots ,z_s, \ldots
\end{pmatrix}
> l^{\star}\begin{pmatrix}
k_1,\ldots, k_s \\
z_1,\ldots ,z_s
\end{pmatrix}
=
l^{\star}\begin{pmatrix}
k_1,\ldots, k_s+1, \{1\}^{\infty} \\
z_1,\ldots ,z_s, \{z_s\}^{\infty} 
\end{pmatrix}
\ge 
l^{\star}\begin{pmatrix}
k_1,\ldots, k_{s-1}, m_{s},\ldots  \\
z_1,\ldots ,z_{s-1}, z_{s}, \ldots  
\end{pmatrix}
\end{align*}
and this means that 
$\eta_{\boldsymbol{z}}(\boldsymbol{k}) >
\eta_{\boldsymbol{z}}(\boldsymbol{m})$.
\end{proof}

The map $\eta_{\boldsymbol{z}}$
is injective by this theorem,
but it is not bijective in general
as we will see in Theorem \ref{thm:mainthm_bijective}.

The following lemma is an analogue of 
\cite[Lemma 3.2]{Li}.
\begin{lemma}\label{lemma:Delta_r}
For $\boldsymbol{k}=(k_1,k_2,\ldots ) \in\mathbb{Z}_{>0}^{\infty}$
and $0<z<1$,
set $ \Delta_r:= 
l^{\star}\begin{pmatrix}
k_1,\ldots, k_r,1\\
z, \ldots , z, z
\end{pmatrix}
-
l^{\star}\begin{pmatrix}
k_1,\ldots, k_r\\
z, \ldots , z
\end{pmatrix}$.
Then $\displaystyle\lim_{r\to\infty} \Delta_r =0$.
\end{lemma}

\begin{proof}
We have
\begin{equation}\label{eq:add_lemma}
\begin{split}
0<\Delta_r 
&= \sum_{m_1\ge \cdots \ge m_{r+1}\ge 2}
\frac{z^{m_1-1}
}{m_1^{k_1}\cdots m_r^{k_r}m_{r+1}}\\
& \le 
\sum_{m_1\ge \cdots \ge m_{r+1}\ge 2}
\frac{z^{m_1-1}}{m_1\cdots m_rm_{r+1}}\\
&=
l^{\star}\begin{pmatrix}
 \{1\}^{r+1}  \\
 \{z\}^{r+1}
\end{pmatrix}
-
l^{\star}\begin{pmatrix}
 \{1\}^{r}  \\
 \{z\}^{r}
 \end{pmatrix}.
\end{split}
\end{equation}
Since 
$\displaystyle \lim_{r\to \infty} 
l^{\star}\begin{pmatrix}
 \{1\}^{r}  \\
 \{z\}^{r}  
\end{pmatrix} $
is convergent, the last equality in
\eqref{eq:add_lemma}
tends to zero as $r\to \infty$.
This proves that 
$\displaystyle\lim_{r\to\infty}\Delta_r=0$.
\end{proof}

\begin{lemma}\label{lemma:non-dense}
When $z_r>z_{r+1}$, we have
$l^{\star}\begin{pmatrix}
k_1,\ldots, k_r\\
z_1, \ldots, z_r
\end{pmatrix}
>
l^{\star}\begin{pmatrix}
k_1,\ldots, k_r+1, \{1\}^{\infty}\\
\boldsymbol{z}
\end{pmatrix}$.

\end{lemma}

\begin{proof}
For $n\ge 1$, we have
\begin{align*}
&l^{\star}\begin{pmatrix}
k_1,\ldots, k_r+1, \{1\}^{n}\\
z_1,\ldots, z_r, \{z_{r}\}^n
\end{pmatrix}
- 
l^{\star}\begin{pmatrix}
k_1,\ldots, k_r+1, \ \ \ \ \ \{1\}^{n}\ \ \ \ \ \\
z_1, \ldots, z_r, \ \ \ \ z_{r+1}, \ldots ,z_{r+n}
\end{pmatrix}\\
&=
\sum_{m_1\ge \cdots \ge m_r \ge \cdots 
m_{r+n}\ge 1}
 \dfrac{z_1^{m_1-1} (\frac{z_2}{z_1})^{m_2-1}
 \cdots (\frac{z_r}{z_{r-1}})^{m_r-1} }
 {m_1^{k_1} m_2^{k_2}\cdots m_r^{k_r+1}  }
\cdot
 \dfrac{1- (\frac{z_{r+1}}{z_r})^{m_{r+1}-1} 
 \cdots (\frac{z_{r+n}}{z_{r+n-1}})^{m_{r+n}-1}
 }
 {m_{r+1}\cdots  m_{r+n} }\\
&=
\dfrac{z_r (1-\frac{z_{r+1}}{z_r})}
{2^{k_1+\cdots +k_r+2}}
+ P_n(z_r, z_{r+1};\boldsymbol{k}).
\end{align*}
Here the first part
comes from the term for 
$m_1=\cdots =m_{r+1}=2$ and $m_{r+2}=\cdots =m_{r+n}=1$
in the sum 
and it does not depend on $n$.
The second remainder part 
$P_n(z_r, z_{r+1};\boldsymbol{k})$ 
is a positive number.
By taking $n\to \infty$, 
we obtain the desired inequality.
\end{proof}

Let us introduce the arrow notation
$\boldsymbol{k}_{\rightarrow}
= (k_1,\ldots ,k_r, 1)$
and $\boldsymbol{k}_{\uparrow}
= (k_1,\ldots ,k_{r-1}, k_r+1)$
for an index $\boldsymbol{k}=(k_1,\ldots, k_r)\in 
\mathbb{Z}_{>0}^r$.
This notation will be used in the proof of the 
following theorem.

\begin{thm}\label{thm:mainthm_bijective}
Let $\boldsymbol{z}=\{z_i\}_{i\ge 1}$ be a sequence
satisfying $1>z_1\ge z_2 \ge \cdots $.
Then the following conditions are equivalent.
\begin{enumerate}
\item 
All $z_i$'s ($i\ge 1$) are equal.

\item 
The map $\eta_{\boldsymbol{z}}$ is bijective.

\item The set $L_{\boldsymbol{z}}:= 
\left\{  
l^{\star}\begin{pmatrix}
k_1,\ldots ,k_r\\
z_1, \ldots ,z_r
\end{pmatrix}  \mid r\ge 1, (k_1,\ldots, k_r) \in 
\mathbb{Z}_{>0}^{r} \right\}$
is dense in $\left(1, \tfrac{1}{1-z_1} \right]$.
\end{enumerate}
\end{thm}

\begin{proof}
First we prove that 
the condition (i) 
includes (ii) and (iii).
We assume that 
$z=z_1=z_2=z_3=\cdots $ and 
denote 
$l^{\star}\begin{pmatrix}
k_1,\ldots, k_r\\
z, \ldots , z
\end{pmatrix}$ by
$l^{\star}(k_1,\ldots, k_r)$
for simplicity.
We also set 
$L(z):= \left\{  l^{\star}(k_1,\ldots, k_r) \mid r\ge 1, k_1,\ldots, k_r \ge 1  \right\} \subset \left( 1, \frac{1}{1-z} \right]$.
For $x\in  \left( 1, \frac{1}{1-z} \right]$, we consider the 
following three cases:
\begin{itemize}

\item 
If $x=\frac{1}{1-z}$, then 
$x=l^{\star} (1,1,\ldots)$.
\item 
If $x\in L(z)$, say $x=
l^{\star}(m_1,\ldots, m_s)$, then
$x= 
l^{\star}(m_1,\ldots, m_s+1, \{1\}^{\infty})$ by
Proposition \ref{prop:limit_prop} (ii).

\item If $x\in \left( 1, \frac{1}{1-z} \right)
\setminus L(z)$, then
because 
\[ 1< \cdots < l^{\star}(3)< l^{\star}(2)< l^{\star}(1)<
l^{\star}(1,1)<l^{\star}(1,1,1)<\cdots <\frac{1}{1-z} \]
and $\lim_{k\to \infty}l^{\star}(k) =1$ and 
$\lim_{m\to \infty}l^{\star}(\{1\}^m) = \frac{1}{1-z}$,
there exists $k\ge 1$ with $l^{\star}(k+1)<x<l^{\star}(k) $
or $m\ge 1$ with $l^{\star}(\{1\}^m)<x<l^{\star}(\{1\}^{m+1})$.

When $l^{\star}(\boldsymbol{k}_{\uparrow})
<x < l^{\star}(\boldsymbol{k})$
with $\boldsymbol{k}\in \mathbb{Z}_{>0}^r$,
there exists $n\ge 1$ such that 
$l^{\star}(\boldsymbol{k}_{\uparrow},\{1\}^{n-1})<x<
l^{\star}(\boldsymbol{k}_{\uparrow},\{1\}^{n})$
because 
$\lim_{n\to\infty}l^{\star}(\boldsymbol{k}_{\uparrow}, \{1\}^n)=
l^{\star}(\boldsymbol{k})$.

When $l^{\star}(\boldsymbol{k})
<x < l^{\star}(\boldsymbol{k}_{\rightarrow})$,
there exists $n\ge 1$ such that 
$l^{\star}(\boldsymbol{k},n+1)
<x< l^{\star}(\boldsymbol{k},n)$
because 
$\lim_{n\to\infty}l^{\star}(\boldsymbol{k}, n)=
l^{\star}(\boldsymbol{k})$.

So there exists an index
$\boldsymbol{k}=(k_1,k_2, \ldots ) \in\mathbb{Z}_{>0}^{\infty}$
such that 
\[ l^{\star}(k_1,\ldots, k_r+1) <x< l^{\star}(k_1,\ldots, k_r)
\text{ or }
l^{\star}(k_1,\ldots, k_r) <x< l^{\star}(k_1,\ldots, k_r,1)\]
holds for infinitely many $r$.
Since $l^{\star}(k_1,\ldots, k_r)
\le l^{\star}(k_1,\ldots, k_{r-1},1)$,
the former condition induces 
$l^{\star}(k_1,\ldots, k_{r-1})
<x< l^{\star}(k_1,\ldots, k_{r-1},1)$.
Consequently, by Lemma \ref{lemma:Delta_r},
the limit value 
$\lim_{r\to \infty} l^{\star}(k_1,\ldots, k_r)$ exists and 
it equals $x$.
\end{itemize}
This means that $\eta_{\boldsymbol{z}}$ is surjective
and the set $L_{\boldsymbol{z}}$ is dense in 
$\left(1, \tfrac{1}{1-z} \right]$.

Next, we prove that 
if the condition (i) is not satisfied,
then 
neither of the two conditions (ii) nor (iii) is satisfied.
Assume that there exists $r\ge 1$ 
such that $z_r>z_{r+1}$.
For an index $\boldsymbol{k}=(k_1,\ldots, k_r)\in\mathbb{Z}_{>0}^r$,
set $w=
l^{\star}\begin{pmatrix}
k_1,\ldots, k_r\\
z_1, \ldots , z_r
\end{pmatrix}$.
By Lemma \ref{lemma:non-dense},
for sufficiently small $\varepsilon>0$,
there is no element of $L_{\boldsymbol{z}}$
in the interval $(w-\varepsilon, w)$.
Hence the map $\eta_{\boldsymbol{z}}$
is not bijective and 
$L_{\boldsymbol{z}}$ is not dense in 
$\left( 1, \frac{1}{1-z_1} \right]$.

\end{proof}

\noindent \textbf{\large Acknowledgment}\\[5pt]
The author would like to thank Professor Minoru Hirose for useful comments. This work was supported by JSPS KAKENHI Grant Number 20K03523.

\vspace{10pt}

\noindent \textbf{Address}: Department of Robotics, Osaka Institute of Technology\\
1-45 Chaya-machi, Kita-ku, Osaka 530-8585, Japan\\

\noindent \textbf{E-mail}: ken.kamano@oit.ac.jp

\end{document}